\documentclass[a4paper,11pt]{amsart}

\usepackage{amsfonts,amssymb}
\usepackage{graphicx,tikz}
\usepackage{float}

\theoremstyle{plain}

\newtheorem*{thmA}{Theorem A}
\newtheorem*{thmB}{Theorem B}
\newtheorem*{thmC}{Theorem C}
\newtheorem{thm}{Theorem}[section]
\newtheorem{lem}[thm]{Lemma}
\newtheorem{pro}[thm]{Proposition}

\theoremstyle{definition}

\newtheorem{dfn}[thm]{Definition}

\newtheorem{exa}[thm]{Example}

\newtheorem{rmk}[thm]{Remark}

\newcommand{\N}{\mathbb{N}}

\DeclareMathOperator{\mci}{mci}

\begin{document}

\title[A restriction on centralizers in finite groups]{A restriction on centralizers\\ in finite groups}
\author[G.A. Fern\'andez-Alcober]{Gustavo A. Fern\'andez-Alcober}
\address{Matematika Saila\\ Euskal Herriko Unibertsitatea UPV/EHU\\
48080 Bilbao, Spain}
\email{gustavo.fernandez@ehu.es}
\author[L. Legarreta]{Leire Legarreta}
\address{Matematika Saila\\ Euskal Herriko Unibertsitatea UPV/EHU\\
48080 Bilbao, Spain}
\email{leire.legarreta@ehu.es}
\author[A. Tortora]{Antonio Tortora}
\address{Dipartimento di Matematica\\ Universit\`a di Salerno\\
Via Giovanni Paolo II, 132\\ 84084 Fisciano (SA)\\ Italy}
\email{mtota@unisa.it}
\author[M. Tota]{Maria Tota}
\address{Dipartimento di Matematica\\ Universit\`a di Salerno\\
Via Giovanni Paolo II, 132\\ 84084 Fisciano (SA)\\ Italy}
\email{antortora@unisa.it}

\thanks{The first two authors are supported by the Spanish Government, grant
MTM2011-28229-C02-02, and by the Basque Government, grants IT460-10 and IT753-13.
The last two authors would like to thank the Department of Mathematics at the University of the Basque Country for its excellent hospitality while part of this paper was being written. They also wish to thank G.N.S.A.G.A. (INdAM) for financial support.}


\begin{abstract}
For a given $m\ge 1$, we consider the finite non-abelian groups $G$ for which
$|C_G(g):\langle g \rangle|\le m$ for every $g\in G \smallsetminus Z(G)$.
We show that the order of $G$ can be bounded in terms of $m$ and the largest prime divisor of the order of $G$.
Our approach relies on dealing first with the case where $G$ is a non-abelian finite
$p$-group.
In that situation, if we take $m=p^k$ to be a power of $p$, we show that $|G|\le p^{2k+2}$ with the only exception of $Q_8$.
This bound is best possible, and implies that the order of $G$ can be bounded by a function of $m$ alone in the case of nilpotent groups.
\end{abstract}

\maketitle

\section{Introduction}

Given a non-abelian group $G$, it makes sense to impose restrictions on the centralizers of non-central elements, and ask what the effect is on the whole of $G$.
For example, we may ask what happens if we require that $|C_G(g)|\le m$ for every
$g\in G\smallsetminus Z(G)$.
In this case, one can quickly bound the orders of the Sylow subgroups of $G$ in terms of $m$: consider separately the cases where a Sylow subgroup $P$ is central or not, and in the latter case, observe that the order of a maximal abelian subgroup of $P$ is bounded.
Consequently the order of $G$ can be bounded by a function of $m$.

A more interesting, but no less natural, restriction arises if we take into account that every element commutes with itself, and put a bound on $|C_G(g):\langle g \rangle|$ as $g$ runs over $G\smallsetminus Z(G)$.
Let us define the \emph{maximum centralizer index\/} of a non-abelian finite group $G$ as
\[
\mci(G)
=
\max \{ |C_G(g):\langle g \rangle| \mid g\in G\smallsetminus Z(G) \}.
\]
Then the goal of this paper is to get bounds for the order of $G$ under the condition that
$\mci(G)=m$.
One cannot bound the order of $G$ by a function of $m$ alone in this case;
for example, if $G$ is non-abelian of order $pq$, with $p$ and $q$ primes, then $m=1$ but $|G|$ is unbounded.
However, as we next show, it is possible to obtain interesting bounds for the order of $G$ by introducing other parameters or by restricting the class of finite groups under consideration.

As with the restriction on centralizers mentioned in the first paragraph, the strategy is to try to bound the orders of the Sylow subgroups.
Thus we begin by considering finite $p$-groups.
Since $\mci(G)$ is obviously a divisor of $|G|$, in this case we have $\mci(G)=p^k$ for some $k\ge 0$.
Rather than contenting ourselves with bounding the order of $G$, we have made an extra effort to obtain the best possible bound in terms of $p$ and $k$.

\begin{thmA}
Let $G$ be a non-abelian finite $p$-group.
If $\mci(G)=p^k$ then $|G|\le p^{2k+2}$, unless $G\cong Q_8$.
This bound is best possible.
\end{thmA}

From this result it readily follows that the order of $G$ can be bounded by a function of $\mci(G)$ alone if $G$ is a non-abelian finite $p$-group.
We note that the proof of Theorem A depends to a great extent on the theory of $p$-central $p$-groups.

In the general case of a finite group, we can use Theorem A to get the following result, where we determine a subset of $\pi(G)$ (the set of prime divisors of the order of $G$) that, together with $\mci(G)$, suffices to bound the order of the group.

\begin{thmB}
Let $G$ be a non-abelian finite group such that $\mci(G)=m$.
If $\pi^*$ is the subset of all primes $p$ in $\pi(G)$ for which a Sylow $p$-subgroup $P$ of $G$ satisfies $|P|=p$ and $C_G(P)=PZ(G)$, then
\[
|G| \le \bigg( \, \prod_{p\in\pi^*}  \, p \, \bigg) \cdot f_0(m)
\]
for some function $f_0$ which depends only on $m$.
\end{thmB}

The subset $\pi^*$ can be analysed with the help of the prime graph of a finite group, which has been extensively studied in the literature.
This leads to the following consequence of Theorems A and B, which shows that very few primes in $\pi(G)$ may escape the control of $\mci(G)$.

\begin{thmC}
Let $G$ be a non-abelian finite group such that $\mci(G)=m$.
Then there exists a function $f_1$ depending only on $m$ such that:
\begin{enumerate}
\item
If $G$ is nilpotent then $|G|\le f_1(m)$.
\item
If $G$ is soluble then $|G|\le D \, f_1(m)$, where $D$ is a product of at most two prime divisors of the order of $G$.
\item
In general $|G|\le  E \, f_1(m)$, where $E$ is a product of at most four prime divisors of the order of $G$.
Thus the order of $G$ can be bounded in terms of $m$ and the largest prime divisor of
$|G|$.
\end{enumerate}
\end{thmC}

\noindent
\textit{Notation.\/}
We use standard notation in group theory.
Also, if $G$ is a finite $p$-group then $\Omega_i(G)$ denotes the subgroup
generated by the elements of $G$ of order at most $p^i$, and $G^{p^i}$
is the subgroup generated by the $p^i$th powers of all elements of $G$.

\section{The case of finite $p$-groups}

This section is devoted to obtaining a bound for the order of a non-abelian finite $p$-group $G$, given that $\mci(G)=p^k$.
Actually, we will get the best possible bound in terms of $p$ and $k$.

\vspace{7pt}

We begin with an easy lemma, where we describe all finite groups $G$ for which
$\mci(G)=1$.

\begin{lem}
\label{mci(G)=1}
Let $G$ be a finite non-abelian group.
Then $C_G(g)=\langle g \rangle$ for every $g\in G\smallsetminus Z(G)$ if and only if
$G\cong Q_8$ or $G$ is non-abelian of order $pq$, where $p$ and $q$ are primes such that $q\equiv 1\pmod p$.
\end{lem}

\begin{proof}
If $G\cong Q_8$ or $G$ is non-abelian of order $pq$ with $p$ and $q$ primes, then it is clear that $\mci(G)=1$.

Let now $G$ be a group such that $\mci(G)=1$, and let $P$ be a Sylow $p$-subgroup of $G$ which is not central in $G$.
If $A$ is a maximal abelian subgroup of $P$, then $A\not\le Z(G)$ and
$C_G(g)=\langle g \rangle$ for every $g\in A\smallsetminus Z(G)$.
Thus we get the following:
\begin{enumerate}
\item
$A$ is cyclic and $|A:A\cap Z(G)|=p$.
\item
If $q\ne p$ is a prime, then $Q\cap C_G(P)=1$ for every Sylow $q$-subgroup $Q$ of $G$.
In particular, $Q\cap Z(G)=1$.
\end{enumerate}

It follows from (i) and \cite[4.4]{suz2} that $P$ is either cyclic or isomorphic to $Q_8$.
In particular, if $G$ is a finite $p$-group then $G\cong Q_8$.
So we assume that the order of $G$ is divisible by at least two primes.
Let $K$ be an arbitrary non-trivial Sylow subgroup of $G$.
By taking (ii) into acount, $K$ is not central in $G$, and so $K$ can play the role of $P$ in the previous paragraph.
Hence $|A:A\cap Z(G)|$ is a prime for every maximal abelian subgroup $A$ of $K$.
But by (ii) above (corresponding to $P$), we know that $K\cap Z(G)=1$.
It follows that every maximal abelian subgroup of $K$ is of prime order, and so $K$ itself is of prime order.

Hence the order of $G$ is square-free.
By \cite[10.1.10]{rob}, $G$ is the semidirect product of two cyclic subgroups of coprime orders.
According to (ii), these two cyclic subgroups must be of prime order.
We conclude that $G$ is a non-abelian group of order $pq$ for two primes $p$ and $q$, as desired.
\end{proof}

A key ingredient in the proof of Theorem A will be the theory of $p$-central $p$-groups.
We recall the definition for the convenience of the reader.

\begin{dfn}
Let $G$ be a finite $p$-group.
We say that $G$ is \emph{$p$-central\/} if $p>2$ and $\Omega_1(G)\le Z(G)$, or if $p=2$ and $\Omega_2(G)\le Z(G)$.
\end{dfn}

Classical references for $p$-central $p$-groups are Buckley's paper \cite{buc} (where they are introduced only for $p>2$, and are called PN-groups) and Laffey's paper \cite{laf}.
We will need results about $p$-central $p$-groups from the recent article \cite{gon-wei} by Gonz\'alez-S\'anchez and Weigel, who deal with a generalization of this class of groups.

\vspace{7pt}

The $p$-central $p$-groups are somehow dual to powerful $p$-groups, which are defined by the condition $G'\le G^p$ if $p$ is odd, or $G'\le G^4$ if $p=2$.
A well-known property of powerful $p$-groups is that
$|G^{p^{i+1}}:G^{p^{i+2}}|\le |G^{p^i}:G^{p^{i+1}}|$ for every $i\ge 0$
(see Theorem 11.15 in \cite{khu}).
We will need the following dual property of $p$-central $p$-groups.

\begin{lem}
\label{consecutive indices}
Let $G$ be a $p$-central $p$-group.
Then $|\Omega_{i+2}(G):\Omega_{i+1}(G)|\le |\Omega_{i+1}(G):\Omega_i(G)|$ for every
$i\ge 0$.
\end{lem}

\begin{proof}
By Theorem B of \cite{gon-wei}, we know that $G/\Omega_i(G)$ is $p$-central, and that
$\exp \Omega_i(G)\le p^i$ for every $i\ge 1$.
Then
\[
| \Omega_2(G/\Omega_i(G)):\Omega_1(G/\Omega_i(G)) |
=
| \Omega_{i+2}(G):\Omega_{i+1}(G) |,
\]
and if we work with $G/\Omega_i(G)$ instead of $G$, it suffices
to prove that $|\Omega_2(G):\Omega_1(G)|\le |\Omega_1(G)|$.
This follows immediately if we see that the map $x\mapsto x^p$ is a homomorphism from
$\Omega_2(G)$ to $\Omega_1(G)$.
This result is obvious if $p=2$, since $\Omega_2(G)$ is then abelian.
If $p>2$ then $\Omega_2(G)\le Z_2(G)$, since $G/\Omega_1(G)$ is $p$-central.
Hence $\Omega_2(G)$ has class at most $2$, and
$\exp \Omega_2(G)' \le \exp \Omega_1(G)\le p$.
Thus $(xy)^p=x^py^p[y,x]^{\binom{p}{2}}=x^py^p$ for every $x,y\in \Omega_2(G)$, and we are done.
\end{proof}

After these preliminary results, we can now prove Theorem A.

\begin{thm}
\label{mci finite p-group}
Let $G$ be a non-abelian finite $p$-group such that $\mci(G)=p^k$.
Then $|G|\le p^{2k+2}$, unless $G$ is isomorphic to the quaternion group $Q_8$.
\end{thm}

\begin{proof}
Assume first that $G$ is $p$-central, and let $r$ be such that $\Omega_r(G)\le Z(G)$ but
$\Omega_{r+1}(G)\not\le Z(G)$.
If $|\Omega_{r+1}(G):\Omega_r(G)|=p$ then
\[
| \Omega_{i+1}(G/\Omega_r(G)): \Omega_i(G/\Omega_r(G)) |
=
| \Omega_{i+r+1}(G):\Omega_{i+r}(G) | \le p
\]
for every $i\ge 0$, by Lemma \ref{consecutive indices}.
It follows that $G/\Omega_r(G)$ is cyclic, and then $G$ is abelian, since
$\Omega_r(G)\le Z(G)$.
Thus we have
\begin{equation}
\label{index at least p^2}
|\Omega_{r+1}(G):\Omega_r(G) | \ge p^2.
\end{equation}
Again by Lemma \ref{consecutive indices}, it follows that
\[
| \Omega_r(G) |
=
\prod_{i=0}^{r-1} \, | \Omega_{i+1}(G):\Omega_i(G) |
\ge
p^{2r-2} \, | \Omega_1(G) |.
\]
Since $|G:G^p|\le |\Omega_1(G)|$ by Theorem C of \cite{gon-wei}, we get
\begin{equation}
\label{bound G with G^p Omega_r}
|G| \le |G^p| \, |\Omega_r(G)| / p^{2r-2}.
\end{equation}

Let us choose an arbitrary element $g\in \Omega_{r+1}(G)\smallsetminus Z(G)$.
Since $\Omega_{r+1}(G)\le Z_2(G)$, we have
\[
[\Omega_{r+1}(G),G^p]
=
[\Omega_{r+1}(G)^p,G]
\le
[\Omega_r(G),G] = 1.
\]
Hence $G^p\le C_G(g)$, and since $\mci(G)=p^k$, we have
\[
p^k
\ge
| C_G(g):\langle g \rangle |
\ge
| G^p \langle g \rangle : \langle g \rangle |
\ge
|G^p| / p^{r+1},
\]
and
\begin{equation}
\label{bound G^p}
|G^p| \le p^{k+r+1}.
\end{equation}
Similarly,
\begin{equation}
\label{mci star p-central}
\begin{split}
p^k
\ge
| C_G(g):\langle g \rangle |
&\ge
| G^p\Omega_r(G) \langle g \rangle : \Omega_r(G) \langle g \rangle |
\,
| \Omega_r(G) \langle g \rangle : \langle g \rangle |
\\
&=
| G^p\Omega_r(G) \langle g \rangle : \Omega_r(G) \langle g \rangle |
\,
| \Omega_r(G) | / p^r,
\end{split}
\end{equation}
and in particular
\begin{equation}
\label{bound Omega_r}
|\Omega_r(G)| \le p^{k+r}.
\end{equation}

Now we consider separately the cases $G^p\not\le \Omega_r(G)$ and
$G^p\le \Omega_r(G)$.
Assume first that $G^p\not\le \Omega_r(G)$.
Since $\Omega_{r+1}(G)/\Omega_r(G)$ is  an elementary abelian $p$-group of order at least $p^2$, and since $Z(G)/\Omega_r(G)$ is a proper subgroup of
$\Omega_{r+1}(G)/\Omega_r(G)$, it follows that
\[
\bigcap_{g\in \Omega_{r+1}(G)\smallsetminus Z(G)} \, \langle g \rangle \Omega_r(G)
=
\Omega_r(G).
\]
Consequently, we can choose $g\in \Omega_{r+1}(G)\smallsetminus Z(G)$ in such a way that $G^p\not\le \langle g \rangle \Omega_r(G)$.
By (\ref{mci star p-central}), we can improve (\ref{bound Omega_r}) to
\begin{equation}
\label{better bound Omega_r}
|\Omega_r(G)|\le p^{k+r-1}.
\end{equation}
On the other hand, if $G^p\le \Omega_r(G)$ then by (\ref{bound Omega_r}) we can improve (\ref{bound G^p}) to
\begin{equation}
\label{better bound G^p}
|G^p| \le p^{k+r}.
\end{equation}
Thus we can combine either (\ref{bound G^p}) and (\ref{better bound Omega_r}), or
(\ref{bound Omega_r}) and (\ref{better bound G^p}), and then use
(\ref{bound G with G^p Omega_r}) to get $|G|\le p^{2k+2}$ in any case.
This completes the proof when $G$ is $p$-central.

Assume now that $G$ is not $p$-central, and suppose that $|G|>p^{2k+2}$.
We are going to prove that $G\cong Q_8$.
Put $\epsilon=0$ or $1$, according as $p>2$ or $p=2$.
Let us choose a subgroup $A$ of $G$ which is maximal in the set of abelian normal subgroups of $G$ of exponent at most $p^{1+\epsilon}$.
By a well-known theorem of Alperin \cite[Chapter III, Theorem 12.1]{hup}, we have
$\Omega_{1+\epsilon}(C_G(A))=A$.
If $A\le Z(G)$ then we get $\Omega_{1+\epsilon}(G)\le Z(G)$, which is not the case.

Thus $(A\cap Z_2(G))\smallsetminus (A\cap Z(G))$ is not empty.
Let $t$ be an arbitrary element in that difference.
Then
\begin{equation}
\label{1st ineq}
p^k \ge | C_G(t):\langle t \rangle | \ge |C_G(t)|/p^{1+\epsilon},
\end{equation}
and in particular
\begin{equation}
\label{2nd ineq}
|A| \le |C_G(t)| \le p^{k+1+\epsilon}.
\end{equation}
Thus
\begin{equation}
\label{3rd ineq}
|A\cap Z(G)| \le |A|/p \le p^{k+\epsilon}.
\end{equation}
On the other hand,
\begin{equation}
\label{4th ineq}
|G:C_G(t)| = | \{ [t,x] \mid x\in G \} | \le |A\cap Z(G)|.
\end{equation}
Consequently
\[
|G| = |G:C_G(t)| \, |C_G(t)| \le p^{2k+1+2\epsilon}.
\]
Since $|G|>p^{2k+2}$, this implies that $p=2$ and $|G|=2^{2k+3}$.
Thus all inequalities in (\ref{1st ineq}), (\ref{2nd ineq}), (\ref{3rd ineq}), and (\ref{4th ineq}) are actually equalities.
It follows that $C_G(t)=A$, and consequently $Z(G)\le A$, that $|A:Z(G)|=2$,
$|A|=2^{k+2}$, and
\begin{equation}
\label{precamina}
Z(G) = \{ [t,x] \mid x\in G \}.
\end{equation}

Since $|A:Z(G)|=2$, we have $A\le Z_2(G)$, and any element of $A\smallsetminus Z(G)$
is a valid choice for $t$.
Also,
\[
[A,G^2] = [A^2,G] \le [Z(G),G] = 1,
\]
and so $G^2\le C_G(t)=A$.
If $g^2\in A\smallsetminus Z(G)$ for some $g\in G$ then we can choose $t=g^2$, and
$g\in C_G(t)\smallsetminus A$, which is a contradiction.
We conclude that $G^2\le Z(G)$.
Since $G'\le G^2$, it follows that $G$ is a group of class $2$, and by (\ref{precamina}), \begin{equation}
\label{camina}
G' = \{ [t,x]\mid x\in G \} = Z(G),
\quad
\text{for every $t\in A\smallsetminus Z(G)$.}
\end{equation}
In particular $|G'|=2^{k+1}$.

On the other hand, we have
\[
\exp Z(G) = \exp G' = \exp G/Z(G) = 2,
\]
by using that $G$ is of class $2$.
Hence $\exp G=4$.
Thus if we choose an arbitrary element $g\in G\smallsetminus G'$, then
$\langle g \rangle Z(G)$ is a normal abelian subgroup of $G$ of exponent at most $4$.
By embedding this subgroup in a maximal abelian normal subgroup of exponent at most $4$, we see that $g$ can play the same role as $t$ above, and in particular
\[
G' = \{ [g,x] \mid x\in G \},
\]
by (\ref{camina}).
Since this holds for every $g\in G\smallsetminus G'$, we conclude that $G$ is a Camina group.
Also $G$ is a special $2$-group, i.e.\ $G'=\Phi(G)=Z(G)$.
We may then apply Theorems 3.1 and 3.2 of \cite{mac}, which are valid for Camina special $p$-groups, to get $|G:G'|\ge |G'|^2$.
Thus
\begin{equation}
\label{last bound for 2}
2^{2k+3} = |G| = |G:G'| \, |G'| \ge |G'|^3 = 2^{3k+3},
\end{equation}
and necessarily $k=0$.
Hence $\mci(G)=1$, and $G\cong Q_8$ by Lemma \ref{mci(G)=1}.
(Alternatively, we get $|G|=8$ from (\ref{last bound for 2}), and so $G\cong Q_8$ or $D_8$.
Since $\mci(Q_8)=1$ but $\mci(D_8)=2$, we necessarily have $G\cong Q_8$.)
\end{proof}

Now we present an example which shows that the bound $|G|\le p^{2k+2}$ in Theorem A is best possible.

\begin{exa}
\label{bound best possible}
Let $p$ be an arbitrary prime, and let $G$ be the group given by the following presentation:
\[
G = \langle a,b \mid a^{p^{k+1}}=b^{\,p^{k+1}}=1,\ a^b=a^{1+p^k} \rangle.
\]
Then $|G|=p^{2k+2}$, $Z(G)=\langle a^p,b^{\,p} \rangle$ and $o(g)=p^{k+1}$ for every
$g\in G\smallsetminus Z(G)$.
By using these facts, one can readily check that $\mci(G)=p^k$.
\end{exa}

\section{The general case}

Now we deal with arbitrary non-abelian finite groups.
We already know that it is not possible to give a general bound for the order of $G$ in terms of $\mci(G)$ alone, and so our goal is to try to obtain bounds by incorporating other parameters.
Contrary to the case of finite $p$-groups, we will not try to get best possible bounds.

\vspace{7pt}

First of all, we use Theorem A to obtain a bound for the order of a non-abelian Sylow subgroup.

\begin{pro}
\label{non-central sylow}
Let $G$ be a non-abelian finite group such that $\mci(G)=m$.
If $P$ is a non-abelian Sylow subgroup of $G$, then $|P|\le 8m^4$.
\end{pro}

\begin{proof}
Let $P$ be a Sylow subgroup for the prime $p$, and let $k$ be the integer part of
$\log_p m$.
Since $\mci(P)\le \mci(G)$ and $p^k\le m<p^{k+1}$, it follows that $\mci(P)\le p^k$.
If $k=0$ then $P\cong Q_8$ by Lemma \ref{mci(G)=1}.
Otherwise we have $p\le m$, and then by applying Theorem A, we get
\[
|P| \le p^{2k+2} = p^2 \cdot p^{2k} \le m^4.
\]
\end{proof}

It immediately follows that a bound in terms of $\mci(G)$ alone exists for nilpotent groups, i.e.\ part (i) of Theorem C.

\begin{thm}
Let $G$ be a nilpotent non-abelian finite group such that $\mci(G)=m$.
Then $|G|\le 8m^5$.
\end{thm}

\begin{proof}
Let $P$ be a non-central Sylow subgroup of $G$.
Since $G$ is nilpotent, we have $G=P\times H$ for some subgroup $H$ of $G$.
Then $P$ is non-abelian, and $|P|\le 8m^4$ by Proposition \ref{non-central sylow}.
On the other hand, since $P$ is non-central and $\mci(G)=m$, it follows that
$|H|\le m$.
\end{proof}

The obstruction to get a bound in terms of $\mci(G)$ alone in the general case is that there may be `bad prime divisors' of the order of $G$ which cannot be bounded by a function of
$\mci(G)$.
Let us make this precise.

\begin{dfn}
Let $G$ be a non-abelian finite group, and let $p$ be a prime divisor of the order of $G$.
We say that $p$ is a \emph{bad prime for mci\/} in $G$ if a Sylow $p$-subgroup $P$ of
$G$ is of order $p$ and $C_G(P)=PZ(G)$.
We write $\pi^*(G)$ for the set of bad primes for mci in $G$.
\end{dfn}

Observe that if $p\in\pi^*(G)$ then $C_G(P)=P\times Z(G)$, since otherwise $G$ is abelian.
The following result is straightforward.

\begin{lem}
\label{bad primes and sbgps}
Let $G$ be a non-abelian finite group, and let $H$ be a non-abelian subgroup of $G$.
If $p$ is a prime divisor of the order of $H$ which is a bad prime for mci in $G$, then it is also a bad prime in $H$.
\end{lem}

We are now ready to prove Theorem B.

\begin{thm}
\label{mci for finite groups}
Let $G$ be a non-abelian finite group such that $\mci(G)=m$.
Then
\begin{equation}
\label{bound mci arbitrary finite}
|G| \le \bigg( \, \prod_{p\in\pi^*(G)}  \, p \, \bigg) \cdot f_0(m),
\end{equation}
where $f_0(m)$ is a function depending only on $m$.
\end{thm}

\begin{proof}
It suffices to prove that, if $p\not\in\pi^*(G)$ and $P$ is a Sylow $p$-subgroup of $G$, then
$|P|$ is bounded by a function of $m$ (take into account that this also implies that the number of primes outside $\pi^*(G)$ is bounded in terms of $m$).

If $P$ is not abelian then $|P|$ is bounded by Proposition \ref{non-central sylow}, so we assume that $P$ is abelian.
In particular, we have $PZ(G)<G$.
Let us choose an arbitrary $q$-element $x\in G\smallsetminus PZ(G)$, for some prime
$q\ne p$.
Then
\begin{equation}
\label{centralizer q-element}
|C_P(x)| = |C_G(x)\cap P:\langle x \rangle \cap P| \le |C_G(x):\langle x \rangle| \le m.
\end{equation}
Now, since $p$ is not a bad prime, we have either $|P|\ge p^2$ or $PZ(G)<C_G(P)$.
In the latter case, we can choose $x$ in $C_G(P)$, and then $|P|=|C_P(x)|\le m$ by (\ref{centralizer q-element}).
Thus we assume that $PZ(G)=C_G(P)$ and $|P|\ge p^2$.
Then $C_P(x)<P$, and we can choose  a subgroup $D$ of $P$ such that $|D:C_P(x)|=p$. If $y\in D\smallsetminus C_P(x)$ then
\begin{equation}
\label{|P:D|}
|P:D|\le |P:\langle y \rangle| \le |C_G(y):\langle y \rangle| \le m,
\end{equation}
since $P$ is abelian, and consequently
\begin{equation}
\label{bound for P}
|P| = |P:D| \, |D:C_P(x)| \, |C_P(x)| \le pm^2.
\end{equation}
If $C_P(x)\ne 1$ then we get $p\le m$ from (\ref{centralizer q-element}).
Otherwise we have $|D|=p$, and consequently $|P:D|\ge p$.
By (\ref{|P:D|}), we also get $p\le m$ in this case.
Thus (\ref{bound for P}) yields that $|P|\le m^3$ in any case, and we are done.
\end{proof}

Our final goal is to see that there are only a few primes in $\pi^*(G)$.
We need a couple of lemmas.
The first one shows that $\pi^*(G)\subseteq \pi^*(G/Z(G))$, and that we can reduce to groups with trivial centre.

\begin{lem}
\label{reduction to Z(G)=1}
Let $G$ be a non-abelian finite group such that $\pi^*(G)$ is not empty.
If $p\in\pi^*(G)$ then the following hold:
\begin{enumerate}
\item
A Sylow $p$-subgroup of $G/Z(G)$ is self-centralizing.
\item
$G/Z(G)$ has trivial centre.
\item
$p\in\pi^*(G/Z(G))$.
\end{enumerate}
\end{lem}

\begin{proof}
Write $\overline G$ for $G/Z(G)$.
Let $P=\langle x \rangle$ be a Sylow $p$-subgroup of $G$, and choose
$\overline y\in C_{\overline G}(\overline P)$ with $\overline y\ne \overline 1$.
Then $[x,y]\in Z(G)$, and consequently $[x,y^p]=[x^p,y]=1$.
Hence $y^p\in C_G(P)=P\times Z(G)$.
Since $y^p$ is a $p'$-element, it follows that $y^p\in Z(G)$.
Thus $\overline y$ is an element of order $p$ in $\overline G$.
Since $\overline y$ centralizes $\overline P$, which is a Sylow $p$-subgroup of
$\overline G$, it follows that $\overline y\in \overline P$.
We conclude that $C_{\overline G}(\overline P)=\overline P$.
If $Z(\overline G)\ne \overline 1$ then necessarily $Z(\overline G)=\overline P$.
Hence $C_{\overline G}(\overline P)=\overline G$ and $\overline G=\overline P$.
This means that $G=PZ(G)$, and so $G$ is abelian, which is a contradiction.
It is now clear that $p\in\pi^*(\overline G)$.
\end{proof}

The second lemma, which can be easily checked, relates the condition of being a bad prime for mci with the \emph{prime graph\/} of $G$.
The set of vertices of this graph is $\pi(G)$, and two primes $p,q\in\pi(G)$ are connected if and only if there exists an element of order $pq$ in $G$.
The prime graph has been extensively studied in the literature; some of the most relevant references are \cite{gru-rog,iiy-yam,iiy-yam2,kon,wil}.
The most important fact about the prime graph is that it has at most six connected components, and actually at most two if the group is soluble.

\begin{lem}
\label{bad primes are isolated}
Let $G$ be a finite group with trivial centre.
Then the following conditions are equivalent for a prime $p$:
\begin{enumerate}
\item
$p$ is a bad prime for mci in $G$.
\item
$p$ divides the order of $G$ only to the first power, and $p$ is an isolated vertex of the prime graph of $G$ (i.e.\ $G$ does not have elements of order $pq$ for any prime
$q\ne p$).
\end{enumerate}
\end{lem}

\begin{thm}
\label{bound for pi star}
Let $G$ be a non-abelian finite group.
Then the following hold:
\begin{enumerate}
\item
$|\pi^*(G)|\le 5$, and if the equality holds then $\pi^*(G)=\{ 23,29,31,37,43 \}$.
\item
If $G$ is soluble then $|\pi^*(G)|\le 2$.
\end{enumerate}
\end{thm}

\begin{proof}
Let us assume that $\pi^*(G)$ is not empty.
By Lemma \ref{reduction to Z(G)=1}, we know that $\pi^*(G)\subseteq \pi^*(G/Z(G))$, and that $G/Z(G)$ has trivial centre.
Then according to Lemma \ref{bad primes are isolated}, a prime $p$ lies in $\pi^*(G/Z(G))$ if and only if $p$ divides $|G/Z(G)|$ only to the first power, and $p$ is an isolated vertex of the prime graph of $G/Z(G)$.

If $G/Z(G)$ is soluble then the prime graph has at most two connected components, and (ii) follows.
Assume now that $G/Z(G)$ is not soluble.
Then $4$ divides the order of $G/Z(G)$, and thus $2$ is not a bad prime for mci in
$G/Z(G)$.
Since the prime graph of $G/Z(G)$ has at most six connected components, it follows that
$|\pi^*(G)|\le 5$.
If the equality holds then necessarily $\pi^*(G)=\pi^*(G/Z(G))$, and the prime graph of
$G/Z(G)$ has exactly six components.
By Theorem B of \cite{zav}, $G/Z(G)$ is isomorphic to $J_4$ the fourth Janko group, and consequently $\pi^*(G)=\pi^*(J_4)$.
Now, according to Table IIa of \cite{wil}, the prime graph of $J_4$ has five isolated vertices, corresponding to the primes $23$, $29$, $31$, $37$, and $43$.
Since all these primes divide the order of $J_4$ only to the first power, it follows that
$\pi^*(J_4)=\{ 23,29,31,37,43 \}$, and we are done.
\end{proof}

\begin{rmk}
One can prove (ii) in the previous theorem without using the fact that the prime graph of an \emph{arbitrary\/} finite soluble group has at most two connected components.
Indeed, assume that $G$ is a finite non-abelian soluble group such that $\pi^*(G)$ is not empty.
Let $R$ be a Hall $\pi^*$-subgroup of $G/Z(G)$, where $\pi^*=\pi^*(G/Z(G))$, and note that all primes in $\pi^*$ are isolated in the prime graph of $G/Z(G)$, by
Lemmas \ref{reduction to Z(G)=1} and \ref{bad primes are isolated}.
Then the order of $R$ is square-free and so by \cite[10.1.10]{rob} we can write
$R=S\ltimes T$, where $S$ and $T$ are cyclic subgroups of $R$ of coprime order.
Since all vertices of the prime graph of $R$ are isolated, it follows that both $S$ and $T$ are of prime order or trivial, and consequently $|\pi^*|\le 2$.
Thus also $|\pi^*(G)|\le 2$.
\end{rmk}

Now parts (ii) and (iii) of Theorem C follow immediately from Theorem B and
Theorem \ref{bound for pi star}.
Simply observe that in the case that $|\pi^*(G)|=5$ the product
$\prod_{p\in\pi^*(G)} \, p$ is a fixed number that can be incorporated to the function depending only on $m$.

\vspace{7pt}

To what extent can part (iii) of Theorem C be sharpened?
Can we get a bound for the order of $G$ with a product of three primes, instead of four?
This question does not have a clear answer, and is connected to deep questions in number theory.
Let us have a closer look at it.
Assume that $|\pi^*(G)|=4$.
Then $|\pi^*(G/Z(G))|\ge 4$, and the same argument used above shows that the prime graph of $G/Z(G)$ has at least five connected components.
If it has six components then $G/Z(G)\cong J_4$, and we get
$\pi^*(G)\subseteq \{23,29,31,37,43\}$.
Now if the prime graph of $G/Z(G)$ has five components, then by Theorem 1 of \cite{zav2} we have $G/Z(G)\cong E_8(q)$ for some prime power $q\equiv 0,\pm 1 \pmod 5$.
By Table Ie of \cite{wil} and Table III of \cite{iiy-yam}, the connected components of the prime graph of $G/Z(G)$ which do not contain $2$ coincide with the sets of prime divisors of the following values:
\begin{equation}
\label{cyclotomic}
\begin{split}
\Phi_{15}(q) &= q^8-q^7+q^5-q^4+q^3-q+1,
\\
\Phi_{20}(q) &= q^8-q^6+q^4-q^2+1,
\\
\Phi_{24}(q) &= q^8-q^4+1,
\\
\Phi_{30}(q) &= q^8+q^7-q^5-q^4-q^3+q+1,
\end{split}
\end{equation}
where $\Phi_n(x)$ stands for the $n$th cyclotomic polynomial over the rationals.
Since $|\pi^*(G)|=4$, it follows that these four connected components must reduce to a single prime, i.e.\ that the four values in (\ref{cyclotomic}) must be prime powers.
The corresponding primes are exactly the bad primes for mci in $E_8(q)$.
Now the values in (\ref{cyclotomic}) are also divisors of the order of $E_8(q)$, and since they are powers of bad primes, it follows that those values are actually prime numbers.
Hence the following question arises: can the cyclotomic polynomials $\Phi_{15}(x)$,
$\Phi_{20}(x)$, $\Phi_{24}(x)$ and $\Phi_{30}(x)$ take simultaneously prime values on a prime power $q$?
If this can only happen for a finite number of choices of $q$, then we could incorporate the corresponding bad primes to the function which depends only on $m$ in the bound of Theorem C.
So we reformulate the previous question, and ask: can the cyclotomic polynomials
$\Phi_{15}(x)$, $\Phi_{20}(x)$, $\Phi_{24}(x)$ and $\Phi_{30}(x)$ take simultaneously prime values on \emph{infinitely many\/} prime powers $q$?
This is a problem in number theory which is connected to the Bunyakovsky conjecture, which asserts that an irreducible polynomial $f(x)$ over the integers such that the values
$\{ f(n)\mid n\in\N \}$ are relatively prime should take infinitely many prime values over the positive integers.
To date this has only been settled (in the positive) for linear polynomials,
by Dirichlet's theorem on arithmetic progressions.

\vspace{20pt}

\noindent
\textit{Acknowledgment\/}.
We thank Alexander Moret\'o for helpful comments regarding the prime graph of a finite group.

\end{document}